\documentclass[12pt]{article}
\usepackage[english]{babel}
\usepackage[utf8]{inputenc}
\usepackage{graphicx,amssymb,latexsym,amsmath,amscd,amsthm}
\usepackage{verbatim,hyperref, color}
\usepackage{cite} 
\usepackage[margin=1.5in]{geometry}

\newcommand{\ra}{\rightarrow} 

\theoremstyle{plain}
\newtheorem{theorem}{Theorem}[section]
\newtheorem{cor}[theorem]{Corollary}
\newtheorem{lemma}[theorem]{Lemma}

\theoremstyle{definition}

\theoremstyle{remark}
\newtheorem{rem}[theorem]{Remark}
\newtheorem{rems}[theorem]{Remarks}

\begin{document} 

\title{The $k$-transformation on an interval with a hole}
\author{Nikita Agarwal\\
\footnotesize Department of Mathematics, \\
\footnotesize Indian Institute of Science Education and Research Bhopal, \\
\footnotesize email: nagarwal@iiserb.ac.in}


\date{}
\maketitle 
\begin{abstract} 
Let $T_{k}$ be the expanding map of $[0,1)$ defined by $T_{k}(x) = k x\ \text{mod 1}$, where $k\geq 2$ is an integer. Given $0\leq a<b\leq 1$, let $\mathcal{W}_{k}(a,b)=\{x\in [0,1)\ \vert \ T_{k}^nx\notin (a,b), \text{ for all } n\geq 0\}$ be the maximal $T$-invariant subset of $[0,1)\setminus (a,b)$. We examine the Hausdorff dimension of $\mathcal{W}_{k}(a,b)$ as $a$ and $b$ vary.
\end{abstract} 


\noindent \textbf{AMS Classification 2010}: 37A05, 28D05

\section{Introduction}
The study of dynamical systems with holes, also termed as \textit{open dynamical systems} was first proposed by Pianigiani and Yorke~\cite{PY}. It has recently attracted attention on account of both its dynamical interest and applications, we refer to~\cite{BKT, BC, BDM, CM, Dem, DemYoung, DWY, HS1, HS2}. Open dynamical systems has several applications including modelling and understanding biological or medical processes, ocean and atmospheric systems, trajectories of spacecraft, planetary motion, see~\cite{Openbook}. 

We now describe the general set-up of an open dynamical system. Let $(X,T)$ be a discrete dynamical system, where $X$ is a compact metric space and $T:X\rightarrow X$ is a continuous map with positive topological entropy. Let $H$ be an open connected subset of $X$, known as the \textit{hole}. The map $T:X\setminus H\ra X$ is called an \textit{open system}, since $X\setminus H$ may not be an invariant set under $T$. Let $\mathcal{W}(H)$ be the maximal $T$-invariant subset of $X\setminus H$. Clearly
\[
\mathcal{W}(H)=\{x\in X\ \vert \ T^nx\notin H, \ n\geq 0\}=X\setminus \bigcup_{n\leq 0}T^n(H) .
\]
The set $\mathcal{W}(H)$ consists of all the points in the state space whose orbit never intersects the hole $H$. This set is called the \textit{survivor set}. 

One can ask several interesting questions regarding the system $T\vert_{\mathcal{W}(H)}$ such as its ergodic properties, see~\cite{BC,CM} and references therein. In~\cite{BY}, the rate at which trajectories escape the hole was considered and it was proved that the escape rate depends not only on the size of the hole but also on its position in the state space. 

The survivor set $\mathcal{W}(H)$ has the following property: for $H, H'\subset X$, $\mathcal{W}(H)\cap \mathcal{W}(H')=\mathcal{W}(H\cup H')$, and thus if $H\subseteq H'$, then $\mathcal{W}(H)\supseteq \mathcal{W}(H')$. The latter property suggests that if the hole is large, the set $\mathcal{W}(H)$ may be countable or empty, whereas it may have positive Hausdorff dimension if the hole is small. The exisiting literature indicates that the size of the hole alone is not responsible for the size of $\mathcal{W}(H)$, its position also matters. 

For the doubling map $T_2$ defined as $T_{2}(x) = 2 x\ \text{mod 1}$ on the interval $[0,1)$, interval holes symmetric about the point $1/2$ were considered in~\cite{GS1}, and asymmetric interval holes were considered in~\cite{GS2}, and the problem of measuring the size of the corresponding survivor set was studied. In~\cite{Clark}, $\beta$-transformation $T_\beta:[0,1)\rightarrow [0,1)$ with any real number $\beta\in(1,2)$ defined as $T_{\beta}(x) = \beta x\ \text{mod 1}$ with an interval hole was considered, and holes were characterized based on whether the suvivor set is non-empty or uncountable. In~\cite{CHS}, the Baker's map with a convex hole was considered and the holes (dimension traps) for which the Hausdorff dimension of the survivor set is zero were examined. It was proved that hole which lies in the interior of the square is not a dimension trap. This work was the first such in higher dimensions. 

The focus of this paper is to examine the Hausdorff dimension of the survivor set for the map $T_k$ defined by $T_{k}(x) = k x\ \text{mod 1}$, ($k\geq 3, k\in\mathbb{N}$) on the interval $[0,1)$ with an interval hole $(a,b)$, where $0\leq a<b\leq 1$. Let $R_1$ and $R_2$ be the collections of intervals contained in $[0,1)$:
\[
R_1 = \left\lbrace(a,b)\ \vert\ b<\dfrac{k-1}{k}\ \text{or}\ a>\dfrac{1}{k}\ \right\rbrace, 
\]
\[
R_2=\left\lbrace(a,b)\ \vert\ a \leq \dfrac{1}{k}\ \text{and}\ b\geq\dfrac{k-1}{k}\ \right\rbrace.
\]

\begin{figure}
		\centering
		\includegraphics[width=.6\textwidth]{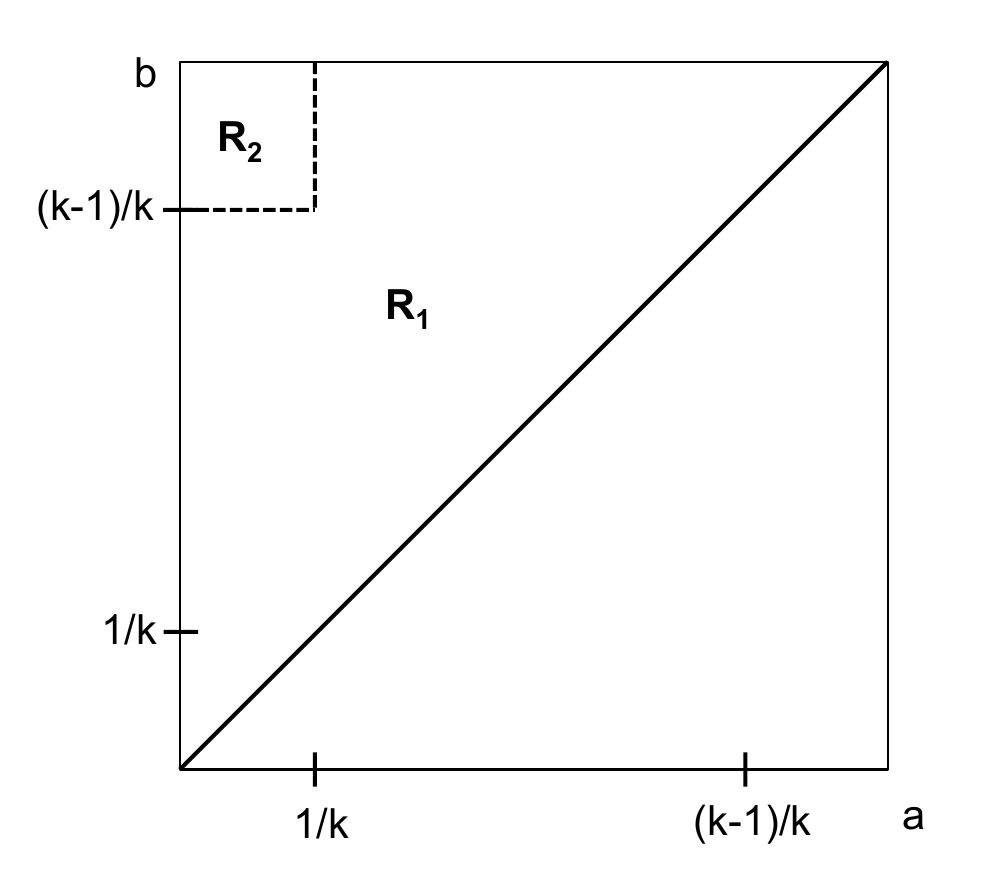}
		\caption{The collections $R_1$ and $R_2$.}
		\label{fig:0}
	\end{figure}	
	
Note that $R_1$ and $R_2$ cover all the possible holes $(a,b)$ with $0\leq a<b\leq 1$. We divide our analysis into two sections based on whether $(a,b)\in R_1$ or $(a,b)\in R_2$. In Section~\ref{R1}, we will prove that the Hausdorff dimension of the survivor set is positive when the hole $(a,b)\in R_1$. In Section~\ref{R2}, we will introduce a generalization of the Cantor set and the Cantor function which will be used to prove the main result (Theorem~\ref{mainthm}) when the hole $(a,b)\in R_2$. In this theorem, we give a necessary and sufficient condition for the Hausdorff dimension of the survivor set to be positive. We will conclude Section~\ref{R2} by describing the relation of our work with earlier results for the doubling map~\cite{GS1,GS2}. In Section~\ref{examgen}, we discuss a possible generalization of $T_{k}$ to higher dimensions.


\subsection{Expanding Map $T_{k}$ with a Hole} 
\noindent Consider the expanding map $T_{k}:[0,1)\rightarrow [0,1)$ with expansion constant $k\geq 2$ ($k\in \mathbb{N}$) defined as 
\[
T_{k}(x) = k x\ \text{mod 1}.
\]
It is well-known that $T_{k}$ is an ergodic map with respect to the one-dimensional Lebesgue measure (see for example~\cite[Proposition 4.4.2]{BS}). Let $H=(a,b)\subset [0,1)$ be an open interval. For $k\geq 3$, we wish to examine the Hausdorff dimension $d_{k}(H)$ of the survivor set $\mathcal{W}_{k}(H) = X\setminus \bigcup_{n\leq 0}T_k ^n(H)$. 

\begin{rem}
a) The results in this paper can be generalized to other kinds of holes. Since we are looking at the map $T_k$ on the circle, the hole $(a,b)\subset [0,1)$ could be taken as $H=(c,1)\cup [0,d)$ where $c<1$, $d\geq 0$, $d=(b+c-a)\ \text{mod }1$. Then $d_k(a,b)=d_k(H)$. Moreover, $x\in \mathcal{W}_k(a,b)$ if and only if $(x+c-a)\ \text{mod }1\in \mathcal{W}_k(H)$.\\
b) Consider an increasing piecewise linear function $h$ on $[0,1)$ which permutes the intervals $\left[ \dfrac{i}{k},\dfrac{i+1}{k}\right)$, $i=0,\dots,k-1$. Then for any $(a,b)\subset [0,1)$ and $H=h(a,b)$ (image of $(a,b)$ under $h$), $d_k(a,b)=d_k(H)$. Moreover $x\in \mathcal{W}_k(a,b)$ if and only if $h(x)\in \mathcal{W}_k(H)$.\\
c) The results in this paper are also applicable to all maps conjugate to the full shift on finitely many symbols.
	\end{rem}
	
	For the doubling map ($k=2$), Glendinning and Sidorov in~\cite{GS1} considered interval holes symmetric about the point $1/2$, and asymmetric interval holes in~\cite{GS2}. In this article, we restrict our attention to integers $k\geq 3$ and consider asymmetric interval holes. We now recall the main result from~\cite{GS2}.

\begin{theorem}\label{asymcor}
\cite[Corollary 3.9]{GS2}
The Hausdorff dimension $d_2(a,b)$ of $\mathcal{W}_{2}(a,b)$ is positive if $b-a<1-2a_*$, where $a_*\sim 0.41245$ is the Thue-Morse constant. 
\end{theorem}

\begin{theorem}\label{asym}
\cite[Theorem 1.2]{GS2}
\begin{eqnarray*}
\left\lbrace(a,b)\in \left(0,1/2\right)\times\left(1/2,1\right)\ \vert\ d_2(a,b)>0\right\rbrace = \left\lbrace(a,b)\ \vert\ b<\chi(a)\right\rbrace,
\end{eqnarray*}
where the function $\chi$ is described in Section~\ref{subsec:chi}.
\end{theorem}

	\subsection{Statement of the Main Results}
We now state the main results of this paper which will be proved in the later sections.

\begin{theorem}\label{m-2}
If $H=(a,b)\in R_1$, then $d_{k}(H)>0$.
\end{theorem}

\begin{theorem}\label{mainthm}
	If $H=(a,b)\in R_2$, the Hausdorff dimension $d_{k}(a,b)>0$ if and only if $g_{k}(b)<\chi(g_{k}(a))$, where the function $\chi$ is described in Section~\ref{subsec:chi}, and $g_k$ is the Cantor function defined in Section~\ref{sec:cantorfunction}. 
\end{theorem}

\section{Background}\label{background}
\subsection{Symbolic Dynamics}
\noindent In this section, we will review few concepts from symbolic dynamics which will be used in this article. We refer to~\cite{BS, Fur67, Kit} for details. \\
For integer $k>1$, let $\Sigma_{k}$ be the set of one-sided sequences with entries from the set $\Lambda_{k}=\{0,1,\cdots,k-1\}$, excluding the sequences ending with $(k-1)^\infty$. For a finite length word $w$ consisting of symbols from $\Lambda_{k}$, we denote its length by $\vert w\vert$. Every such finite word $w$ can be represented as $w0^\infty\in \Sigma_{k}$. Set
\begin{equation}\label{bad_set}
B_{k} =\left\{\frac{\ell}{k^n}\ \vert\ \ell=0,1,\cdots,k^n-1,\ n\in \mathbb{N}\right\}.
\end{equation}
\noindent Let $\sigma_{k}:\Sigma_{k}\rightarrow \Sigma_{k}$ be the one-sided shift map defined as 
\[
\sigma_{k}(a_1,a_2,a_3,\cdots)=(a_2,a_3,\cdots).
\] 
We identify $\Sigma_k$ with the interval $[0,1)$ via the map $\pi_{k}:\Sigma_{k}\rightarrow [0,1)$ defined as
\[
\pi_{k}(a_1,a_2,a_3,\cdots) = \sum_{n=1}^\infty \dfrac{a_n}{k^n},
\]
for $(a_1,a_2,a_3,\cdots)\in\Sigma_k$. \\
The map $\pi_k$ is a bijection, and the inverse image of any element of $B_k$ is a sequence in $\Sigma_k$ ending with $0^\infty$. Representations of real numbers with an arbitrary base $k>1$ were introduced by R\'{e}nyi~\cite{Ren}, and is called its $k$-expansion. Here $\pi_k^{-1}x\in \Sigma_k$ gives the $k$-expansion of $x\in [0,1)$. Note that the points (in $B_k$) have two $k$-expansions, one ending with $0^\infty$, and other ending with $(k-1)^\infty$. Further the diagram given below commutes:
\begin{eqnarray}\label{maincom}
\begin{CD}
\Sigma_{k} @>\sigma_{k}>> \Sigma_{k}\\
@VV\pi_{k} V@VV \pi_{k} V\\
[0,1) @>T_{k}>> [0,1).
\end{CD}
\end{eqnarray}
That is, $T_{k}\pi_{k}=\pi_{k}\sigma_{k}$, for all $k\geq 2$.\\
\noindent A partial order $\prec$ can be defined on $\Sigma_{k}$ as follows: $u\prec v$ if and only if either $u_1<v_1$, or there exists $\ell\geq 2$ such that $u_i=v_i$, for $i=1,\cdots,\ell-1$, and $u_{\ell}<v_\ell$. For $u,v\in\Sigma_k$, we denote the set of all sequences $w\in\Sigma_k$ such that $u\prec w$ and $w\prec v$, including $u$ and $v$, by $[u,v]$, which is called an interval. 

\begin{lemma}
If $x,y\in [0,1)\setminus B_k$ with $x<y$ then $\pi_{k}^{-1}x\prec \pi_{k}^{-1}y$.
\end{lemma}

%



	\subsection{The function $\chi$}\label{subsec:chi}
	The function $\chi$ given in Theorems~\ref{asym} and~\ref{mainthm} was introduced in~\cite{GS2}. The definition of this function requires a few notations which we will explain now.
	
\begin{itemize}
\item Let $r=p/q\in (0,1/2)$ has continued fraction expansion $[d_1+1,d_2,\dots,d_n]$. Since $r<1/2$, $d_1\geq 1$ and $d_n\geq 2$. Define the binary sequence given by $r$ as follows:
\[
u_{-1}=1,\ u_0=0,\ u_{k+1}=u_k^{d_{k+1}}u_{k-1},\ 0\leq k\leq n-1.
\] 
The word $u_n$ is called the \emph{$n^{th}$ standard word given by $r$} with length $q$. 
\item For an irrational $\gamma\in (0,1/2)$ having continued fraction expansion $[d_1+1,d_2,\dots,]$. Let $u_n$ be as defined above. The limit $u_{\infty}$ of $u_n$ as $n\rightarrow \infty$ is called the \emph{characteristic word given by $\gamma$}. 
\item For $r=p/q\in (0,1/2)$, $\rho_r$ defined on the symbols $0$ and $1$ is given by
\[
\rho_r(0)=01w_1\dots w_{q-2},\ \rho_r(1)=10w_1\dots w_{q-2}, 
\]
where the $n^{th}$ standard word given by $r$ is of the form $u_n=w_1\dots w_q$.
\item Let $r\in \mathbb{Q}\cap (1/2,1)$,  
\[
\rho_r(0)=h(\rho_{1-r}(1)),\ \rho_r(1)=h(\rho_{1-r}(0)), 
\]
where $h(0)=1$, $h(1)=0$, and $h(w_1\dots w_n)=h(w_1)\dots h(w_n)$.
\item For any finite vector $(r_1,\dots,r_n)\in (\mathbb{Q}\cap (0,1))^n$, 
\[
\Delta(r_1,\dots,r_n)=[s_n^\infty,s_nt_n^\infty],\ \widetilde{\Delta}(r_1,\dots,r_n)=[s_nt_ns_n^\infty,s_nt_n^\infty],
\]
where 
\[
s_n=\rho_{r_1}\dots \rho_{r_n}(0), \ t_n=\rho_{r_1}\dots \rho_{r_n}(1). 
\]
\item If $r_i=p_i/q_i$, then $Q_n=q_1\dots q_n$. Note that the length of both $s_n$ and $t_n$ is $Q_n$.
\item Let $\mathcal{S}$ be the collection of points in $(1/4,1/2)$ whose binary expansion is of the form $01w$, where $w$ is a characteristic word for some irrational number $\gamma\in (0,1/2)$.
\item For $n\geq 2$, let $\mathcal{S}_n(r_1,\dots,r_{n-1})$ be the collection of points in $(1/4,1/2)$ whose binary expansion is of the form $s_{n-1}t_{n-1}w$, where $w$ is a characteristic word for some irrational number $\gamma\in (0,1/2)$ with $0$ replaced by $s_{n-1}$, and 1 replaced by $t_{n-1}$.
\item For an infinite vector $(r_1,r_2,\dots)\in (\mathbb{Q}\cap (0,1))^\mathbb{N}$,
\[
s(r)=\lim_{n\rightarrow\infty} s_n,\ t(r)=\lim_{n\rightarrow\infty} t_n.
\]
\end{itemize}

\begin{theorem} 
\cite[Theorem 2.13]{GS2} A real number $a\in (1/4,1/2)$ falls into one of the following four categories:
\begin{enumerate}
\item For all $a\in [s_n^\infty,s_nt_ns_n^\infty]$,  $\chi(a)=t_ns_n^\infty$. Furthermore, for any $a<s_n^\infty$, $\chi(a)<t_ns_n^\infty$, and for any $a>s_nt_ns_n^\infty$, $\chi(a)>t_ns_n^\infty$.
\item For all $a\in\mathcal{S}$, $\chi(a)=a+1/4$.
\item For $n\geq 2$ and for all $a\in\mathcal{S}_n(r_1,\dots,r_{n-1})$, 
\[
\chi(a)=a+(1-2^{-Q_{n-1}})(t_{n-1}-s_{n-1}).
\]
\item If there exists $(r_1,r_2,\dots)\in (\mathbb{Q}\cap (0,1))^\mathbb{N}$ such that for all $n\geq 1$, $a\in \widetilde{\Delta}(r_1,\dots,r_n)=[s_nt_ns_n^\infty,s_nt_n^\infty]$, then $a=s(r)$ and $\chi(a)=t(r)$.
\end{enumerate}
\end{theorem}

\begin{theorem} 
\cite[Proposition 3.3]{GS2} For all $a\in (1/4,1/2)$,
\[
a+1-2a_*\leq \chi(a) \leq a+1/4,
\]
where $a_*=\lim_{n\rightarrow\infty} \rho^n_{1/2}(0)$ is the Thue-Morse constant. Furthermore, both of these bounds are sharp. The lower one is attained at $a=a_*$ (that is, when $r=(1/2,1/2,\dots)$, while the upper one is attained when $a\in\mathcal{S}$.
\end{theorem}

\section{Results for $(a,b)\in R_1$} \label{R1}
\noindent In this section, we prove Theorem~\ref{m-2} which states that for $k\geq 3$ and $H=(a,b)$, the Hausdorff dimension $d_k(H)$ is positive when either $H\subseteq \left[0,\dfrac{k-1}{k}\right]$, or $H\subseteq \left[\dfrac{1}{k},1 \right]$. Further, Theorem~\ref{lb} gives an explicit lower bound for $d_k(H)$ when $H\subseteq \left[0,\dfrac{k-j}{k}\right]$ or $H\subseteq \left[\dfrac{j}{k},1\right]$, $j=2,\cdots,m$.
%


\begin{lemma}\label{invert}
$x\in \mathcal{W}_{k}(a,b)$ if and only if $1-x\in \mathcal{W}_{k}(1-b,1-a)$. Thus, $d_{k}(a,b)=d_{k}(1-b,1-a)$. 
\end{lemma}
\begin{proof}
	Observe the following equalities:
\begin{eqnarray*}
x\in \mathcal{W}_{k}(a,b) &\iff& T_{k}^n(x)\notin (a,b),\ \forall n\geq 0 \\
&\iff& 1-T_{k}^n(x) \notin (1-b,1-a),\ \forall n\geq 0 \\
&\iff& T_{k}^n(1-x)\notin (1-b,1-a),\ \forall n\geq 0\\
&\iff& 1-x\in \mathcal{W}_{k}(1-b,1-a).
\end{eqnarray*}
\end{proof}


\begin{proof}(Proof of Theorem~\ref{m-2})
Consider first the case when $H\subseteq \left[0,\dfrac{k-1}{k}\right)$. The other case is similar. For $n\geq 1$, define 
\[
U_n=\{(w_i)\in \{k -2,k -1\}^{\mathbb{N}}\cap \Sigma_k \ \vert \ w_\ell=k-2 \Rightarrow w_{\ell+j}=k-1,\ j=1,\cdots,n\}.
\]
Then $U_n$ is a $\sigma_{k}$-invariant set. We will prove the result in two steps:\\

\noindent We first prove that there exists $N\geq 1$ such that $\pi_{k}(U_N)\subseteq \mathcal{W}_{k}(H)$. Since $b<\frac{k-1}{k}$, there exists $N\in \mathbb{N}$ such that $b<\frac{k-1}{k}-\frac{1}{k^N}$. For $w=(w_i)\in U_n$, either 
\[
\pi_{k}(w)\geq \frac{k-1}{k}, \ \text{or}
\] 
\[
\pi_{k}(w)\geq \frac{k-2}{k}+\frac{k-1}{k^2}\left(1+\frac{1}{k}+\cdots+\frac{1}{k^{N-1}}\right)=\frac{k-2}{k}+\frac{k^N-1}{k^{N+1}}.
\]
Hence $\pi_{k}(w)\in \mathcal{W}_{k}(H)$, for all $w\in U_N$.\\~\\
From standard arguments using adjacency matrix of a suitable graph and Perron-Frobenius theorem, one can prove that the topological entropy of $U_N$ is positive, and hence the Hausdorff dimension of $\pi_k(U_N)$ is positive. Since $\pi_{k}(U_N)\subseteq \mathcal{W}_{k}(H)$, we have $d_k(H)>0$.\\ 
Further for $H\subseteq \left[\dfrac{1}{k},1\right)$, apply Lemma~\ref{invert}.
\end{proof}

\begin{theorem}~\label{lb}
For $j=2,\cdots,m$, if $H\subseteq \left[0,\dfrac{k-j}{k}\right)$ or $H\subseteq \left[\dfrac{j}{k},1\right)$, then $d_{k}(H)\geq\log_k j$.
\end{theorem}

\begin{proof}
If $H\subseteq \left[0,\dfrac{k-j}{k}\right)$, it is easy to check that $\pi_{k}(\{k-j,\cdots,k-1\}^\mathbb{N}\cap\Sigma_k)\subseteq \mathcal{W}_{k}(H)$. Thus, 
\begin{eqnarray*}
d_{k}(H) &\geq& d_{k}(\pi_{k}(\{k-j,\cdots,k-1\}^\mathbb{N}\cap\Sigma_k)) \\
&=&  \dfrac{h_{top}(\{k-j,\cdots,k-1\}^\mathbb{N})}{\log k} = \dfrac{\log j}{\log k}.
\end{eqnarray*}
For $H\subseteq \left[\dfrac{j}{k},1\right)$, apply Lemma~\ref{invert}.
\end{proof}


\noindent The remaining types of holes $H$ belong to the region $R_2$. This will be considered in Section~\ref{R2}. We will discuss preliminaries required for the main result Theorem~\ref{mainthm} for this remaining set of holes $(a,b)$, where $a<\dfrac{1}{k}$ and $b>\dfrac{k-1}{k}$, in Section~\ref{cantor}. For such a hole $(a,b)$, the Hausdorff dimension $d_k(a,b)$ of $\mathcal{W}_{k}(a,b)$ relates to the Hausdorff dimension $d_2(a,b)$ of $\mathcal{W}_2(a,b)$ (doubling map).

\section{The remaining holes $(a,b)\in R_2$}
\label{R2}

\noindent In this section, we will focus on the remaining holes, that is, when $a\leq \dfrac{1}{k}$ and $b\geq\dfrac{k-1}{k}$. A straightforward lemma is as follows.

\begin{lemma}
If $a<\dfrac{1}{k^2}, b>\dfrac{k-1}{k}$, or $a<\dfrac{1}{k}, b>\dfrac{k^2-1}{k^2}$, then $\mathcal{W}_k(a,b)=\{0,1\}$.
\end{lemma}
\begin{proof}
Let us consider the case $a<\dfrac{1}{k^2}, b>\dfrac{k-1}{k}$. The other case follows from Lemma~\ref{invert}.\\
Let $x\in \mathcal{W}_k(a,b)$. Since $a<\dfrac{1}{k^2}$ and $b>\dfrac{k-1}{k}$, $\pi_k^{-1}(x)$ cannot contain $0w$, $w\in\{1,\dots,k-1\}$, and $1,\dots,k-2$. Also, if $\pi_k^{-1}(x)$ ends with $(k-1)0^\infty$, then some iterate of $x$ under $T_k$ lies in $(a,b)$, which is a contradiction. Hence $\mathcal{W}_k(a,b)=\{0,1\}$. 
\end{proof}

It turns out that the analysis for holes in $R_2$ is identical to corresponding holes under $g_{k}$ for the doubling map $T_2$, see Theorem~\ref{mainthm}. The analysis for such holes is similar to the doubling map, via the generalized Cantor set and Cantor function, which we will introduce in the following subsections.

\subsection{Generalization of Cantor Set and Cantor Function}
\label{cantor}
\noindent In this section, we define a set $C_{k}$ and a function $g_{k}:[0,1]\rightarrow [0,1]$ which are generalizations of the Cantor set and the Cantor function, respectively. The function $g_{k}$ maps $C_{k}$ onto the unit interval $[0,1]$. 

\subsubsection{The Cantor set $C_{k}$}
Divide the unit interval $[0,1]$ into $k$ equal sub-intervals. Let $I_1$ denote the union of closed intervals $I_{11}=\left[0,\dfrac{1}{k}\right]$ and $I_{12}=\left[\dfrac{k-1}{k},1\right]$ obtained by removing the middle $(k-2)$ sub-intervals. Repeat the same procedure for each of the intervals $I_{11}$ and $I_{12}$ to obtain sub-intervals $I_{21}, I_{22}$ of $I_{11}$ and $I_{23}, I_{24}$ of $I_{12}$. Set $I_2=I_{21}\cup I_{22}\cup \cup I_{23}\cup I_{24}$. Thus, at the $\ell^{th}$ step, we obtain $2^\ell$ intervals each of length $\dfrac{1}{k^\ell}$, whose union is $I_k$, say. Moreover, $(I_\ell)$ is a decreasing sequence of closed sets. We define the generalized Cantor set as
\[
C_{k} =\cap_{\ell\geq 1} I_\ell.
\]
\begin{rem} 
	$C_{k}$ is a closed set with zero Lebesgue measure. The Hausdorff dimension of $C_{k}$ is $\log_{k} 2$. We refer to~\cite{BS} for details.
\end{rem} 

\begin{rem}
	$\pi_{k}(\{0,k-1\}^\mathbb{N}\cap\Sigma_k)=C_{k}$. Thus, by the commuting diagram~\eqref{maincom}, $T_{k}(C_{k})=C_{k}$.
\end{rem} 

\subsubsection{The Cantor function $g_{k}$}\label{sec:cantorfunction}
For $x\in[0,1]$, define 
\[
g_{k}(x)=\sum_{n=1}^{N-1} \dfrac{a_n}{(k-1)2^n}+\dfrac{1}{2^{N}},
\]
where $x=\sum_{n=1}^\infty  \dfrac{a_n}{k^n}$,
and $N\geq 1$ is the least index such that $a_{N}\notin \{0,k-1\}$. If each $a_n\in\{0,k-1\}$, define
\[
g_{k}(x)=\sum_{n=1}^\infty \dfrac{a_n}{(k-1)2^n}.
\]
The function $g_{k}$ is well-defined. Recall that 
\[
B_{k}=\left\lbrace \dfrac{\ell}{k^n}\ \vert \ \ell=1,\cdots,k^n-1, n\in\mathbb{N}\right\rbrace.
\]
Each $x\in [0,1)\setminus B_{k}$ has a unique $k$-expansion, that is, it has a unique pre-image under $\pi_{k}$. \\
If $x\in B_{k}$, then
$x$ has two $m$-expansions given by $x_1$ and $x_2$ below.
\[
x_1=\sum_{n=1}^p \dfrac{a_n}{k^n}+\sum_{n=p+2}^\infty \dfrac{k-1}{k^n},\ \text{and}\ 
x_2=\sum_{n=1}^p \dfrac{a_n}{k^n}+\dfrac{k-1}{k^{p+1}}, 
\]
for some $a_1,\cdots,a_p\in\Lambda_{k}$. \\
If $a_1,\cdots,a_p\in\{0,k-1\}$, then  
\[
g_{k}(x_1)=\sum_{n=1}^p \dfrac{a_n}{(k-1)2^n}+\sum_{n=p+2}^\infty \dfrac{1}{2^n},\ \text{and}\
g_{k}(x_2)=\sum_{n=1}^p \dfrac{a_n}{(k-1)2^n}+\dfrac{1}{2^{p+1}}.
\]
Else, if $1\leq N\leq p$ is the least index such that $a_{N}\notin \{0,k-1\}$, then 
\[
g_{k}(x_1)=g_{k}(x_2)=\sum_{n=1}^{N-1} \dfrac{a_n}{(k-1)2^n}+\dfrac{1}{2^{N}}.
\]
Hence $g_{k}(x_1)=g_{k}(x_2)$ in both the cases. Therefore, $g_{k}$ is well-defined. The following are some properties of the function $g_{k}$.\\~\\
\textbf{Properties of $g_{k}$}:
\begin{enumerate}
	\item $g_{k}$ maps $C_{k}$ onto $[0,1)$ by construction.
	\item $g_{k}$ is an increasing and continuous function.
	\item $g_{k}$ is constant on each interval of the form 
\[
I_{a_1\cdots a_N}=[\pi_{k}(a_1\cdots a_N10^\infty),\pi_{k}(a_1\cdots a_N(k-1)0^\infty)],
\]
where $a_i\in\{0,k-1\}$, for $i=1,\cdots,N$. Note that 
\[
g_{k}(x)=\sum_{n=1}^N \dfrac{a_n}{(k-1)2^n}+\dfrac{1}{2^{N+1}},
\]
for all $x\in I_{a_1\cdots a_N}$. 
	\item $g_{k}$ is differentiable with $g'=0$ on  $[0,1)\setminus C_{k}$.
	\item $g_{k}^{-1}([0,1]\setminus B_2)=C_{k}\setminus B_{k}$, and there is a one-to-one correspondence under $g_{k}$ between $C_{k}\setminus B_{k}$ and $[0,1)\setminus B_2$. 
	\item $g_k$ is a H\"{o}lder continuous function of exponent $s$ provided $s\leq \log_k2$.
\end{enumerate}	

\begin{rems}
The above properties follow from the properties of the standard Cantor function (with $k=3$), see for instance~\cite{HT}.
\end{rems}

\begin{lemma}\label{gtilde}
The map $g_{k}$ induces a map $\widetilde{g_{k}}:\Sigma_{k}\rightarrow \Sigma_2$ such that the following diagram commutes:
\begin{eqnarray*}
	\begin{CD}
	\Sigma_{k} @>\widetilde{g_{k}}>> \Sigma_2\\
	@VV\pi_{k} V @VV \pi_2 V\\
	[0,1) @>g_{k}>> [0,1).
	\end{CD}
	\end{eqnarray*}
	Moreover, 
	\begin{enumerate}
	\item $\widetilde{g_{k}}$ is constant on each interval of the form 
\[
\widetilde{I}_{a_1\cdots a_N}=[a_1\cdots a_N10^\infty, a_1\cdots a_N(k-1)0^\infty],
\]
where $a_i\in\{0,k-1\}$, for $i=1,\cdots,N$. Note that 
\[
\widetilde{g_{k}}(w)=\dfrac{a_1}{k-1}\cdots \dfrac{a_N}{k-1}01^\infty,
\]
for all $w\in \widetilde{I}_{a_1\cdots a_N}$.
	\item $\widetilde{g_{k}}$ is an increasing function with partial order $\prec$.
	\item $\widetilde{g_{k}}$ maps $\{0,k-1\}^\mathbb{N}$ onto $\Sigma_2$ as 
	\[
	\widetilde{g_{k}}(a_1,a_2,\cdots)=\left(\dfrac{a_1}{k-1},\dfrac{a_2}{k-1},\cdots\right).
	\]
\end{enumerate}	
\end{lemma}

\begin{proof} 
The result follows from the definition of $\widetilde{g_k}$ using the commuting diagram.
\end{proof}

\subsection{Results}
\begin{lemma}\label{B} 
For the doubling map $T_2$, $a<\dfrac{1}{2}$ and $b>\dfrac{1}{2}$, $\mathcal{W}_2(a,b)\cap B_2=\emptyset$, where $B_2$ is the set of dyadic rationals in $[0,1)$, see~\eqref{bad_set}.
\end{lemma}
\begin{proof}
	Let $x\in B_2$, then $u=\pi_2^{-1}x=w01^\infty=w10^\infty$, where $w$ is a finite word consisting of symbols $0,1$. Then $\sigma^{\vert w\vert}u=01^\infty$. Hence $T_2^{\vert w\vert}x\in(a,b)$. Thus, $x\notin \mathcal{W}_2(a,b)$. 
\end{proof}

\begin{lemma}\label{lem}
	The following diagram commutes:
	\begin{eqnarray*}
	\begin{CD}
	C_{k} @>T_{k}>> C_{k}\\
	@VVg_{k} V @VV g_{k} V\\
	[0,1) @>T_2>> [0,1).
	\end{CD}
	\end{eqnarray*}
\end{lemma}

\begin{proof} 
The proof is straightforward.
\end{proof}

\begin{lemma}\label{mainlemma}
	For $a<\dfrac{1}{k}$ and $b>\dfrac{k-1}{k}$, 
	\begin{enumerate}
	\item[a)] $\mathcal{W}_{k}(a,b) \subseteq C_{k}$,
	\item[b)] $g_{k}^{-1}\mathcal{W}_2(g_{k}(a),g_{k}(b))\cap B_{k}=\emptyset$, 
	\item[c)] $\mathcal{W}_{k}(a,b)= C_{k}\cap g_{k}^{-1}\mathcal{W}_2(g_{k}(a),g_{k}(b))$, and
	\item[d)] $\pi_2^{-1}\mathcal{W}_{2}(g_{k}(a),g_{k}(b))= \{w/(k-1) \ \vert\ w\in \pi_k^{-1}\mathcal{W}_{k}(a,b)\}$.
	\end{enumerate}
\end{lemma}
\begin{proof}
	a) If $x\in \mathcal{W}_{k}(a,b)$, then $v=\pi_{k}^{-1}x$ cannot contain the symbols $1,\cdots,k-2$. Since if $v_r\notin \{0,k-1\}$, then $\pi_{k}\sigma^rv\in (a,b)$. Hence $T_{k}^rx\in (a,b)$, which is a contradiction. Thus $\mathcal{W}_{k}(a,b)\subseteq C_{k}$.\\
	b) If $a\in I_0$ and $b\in I_{k-1}$, then $g_{k}(a)<\dfrac{1}{2}$ and $g_{k}(b)>\dfrac{1}{2}$. \\
	Hence, by Lemma~\ref{B}, $\mathcal{W}_2(g_{k}(a),g_{k}(b))\cap B_2=\emptyset$. Therefore, 
	\[
	g_{k}^{-1}\mathcal{W}_2(g_{k}(a),g_{k}(b))\cap B_{k}=\emptyset.
	\]  
c) It follows from Lemmas~\ref{gtilde} and~\ref{lem}.\\
d) It follows from parts a), b), and c).
\end{proof}

%
\begin{proof}(Proof of Theorem~\ref{mainthm})
From Lemma~\ref{mainlemma} d), 
	\[
	d_2(g_{k}(a),g_{k}(b)) = \dfrac{1}{\log_k 2}d_k(a,b).
	\] 
	Hence by Theorem~\ref{asym}, $d_{k}(a,b)>0$ if and only if $g_{k}(b)<\chi(g_{k}(a))$.\\
	Note that $g_k$ is a H\"{o}lder continuous function of exponent $\log_k2$, hence
	\[
	d_2(g_{k}(a),g_{k}(b))\leq \dfrac{1}{\log_k2}d_k(a,b).
	\] 
	Hence it follows by Theorem~\ref{asym} that $d_{k}(a,b)>0$ if $g_{k}(b)<\chi(g_{k}(a))$.
\end{proof}

\begin{cor}
	(Notation as before) \\
	$d_{k}(a,b)$ is constant for all $(a,b)$ in the square $I_{0a_1\cdots a_N}\times I_{(k-1)b_1\cdots b_M}$.
\end{cor}
\begin{proof}
	This is an immediate consequence of Theorem~\ref{mainthm} and Property (3) of $g_k$.
\end{proof}

\begin{rem}
	Glendinning and Sidorov~\cite{GS2} showed that $d_2(a,b)>0$, if $b-a<1-2a_{*}$, where $\pi_2^{-1}(a_{*})=\overline{0}\ \overline{1}\ \overline{10}\ \overline{1001}\ \cdots$, which is known as the Thue-Morse sequence, and $a_{*}\sim 0.41245$ is known as the Thue-Morse constant. An immediate consequence of Theorems~\ref{mainthm} and~\ref{asymcor} is Corollary~\ref{bound}. It is worth noting that 
	\[
	\pi_{k}^{-1}g_{k}^{-1}(a_{*})=\overline{0}\ \overline{k-1}\ \overline{(k-1)0}\ \overline{(k-1)00(k-1)}\ \cdots, 
	\]
	and in particular, $g_3^{-1}(a_*)\sim 0.3049$ and $g_4^{-1}(a_*)\sim 0.2374$. Table~\ref{table:1} shows various values of $g_{k}^{-1}(a_{*})$ as $k$ varies.
\end{rem}

\begin{cor}\label{bound}
	The Hausdorff dimension $d_{k}(a,b)>0$ if $b-a<1-2g_{k}^{-1}(a_{*})$. That is, any hole $H=(a,b)$ with size less than $1-2g_{k}^{-1}(a_{*})$ has $d_{k}(H)>0$. See Table~\ref{table:1}.
\end{cor}

\begin{table}[h!] 
	\centering
	\caption{Size of Hole.}\label{table:1}
	\begin{tabular}{|c|c|c|} 
		\hline\hline
$\textbf{k}$ & $\mathbf{g_{k}^{-1}(a_{*})\sim}$&$\mathbf{1-2g_{k}^{-1}(a_{*})\sim}$	\\\hline \hline
2 & 0.4124 & 0.1751	\\\hline
3 & 0.3049 & 0.3901	\\\hline
4 & 0.2374 & 0.5253	\\\hline
5 & 0.1933 & 0.6134	\\\hline
6 & 0.1627 & 0.6746	\\\hline
7 & 0.1403 & 0.7194	\\\hline
8 & 0.1233 & 0.7535	\\\hline
9 & 0.1099 & 0.7802	\\\hline
10 & 0.09909 & 0.8018	\\\hline
	\end{tabular}
\end{table}

\begin{rem}
In \cite{RB}, the authors discuss about the holes $(a,b)$ when $T\vert_{\mathcal{W}_2(a,b)}$ is conjugate to a subshift of finite type. Let 
\[
D=\left\lbrace(a,b)\in \left(0,1/2\right)\times\left(1/2,1\right)\ \vert\ d_2(a,b)>0\right\rbrace.
\] 
They show that the set 
\[
\{(a,b)\in D \ \vert\ T_2\vert_{\mathcal{W}_2(a,b)} \text{ is conjugate to a subshift of finite type}\}
\]
is open, dense and has full Lebesgue measure in $D$. A similar result will be true here using Theorem~\ref{mainthm}. If
\[
D'=\left\lbrace(a,b)\in \left(0,1/k\right)\times\left((k-1)/k,1\right)\ \vert\ d_k(a,b)>0\right\rbrace,
\] 
then the set 
\[
\{(a,b)\in D' \ \vert\ T_k\vert_{\mathcal{W}_k(a,b)} \text{ is conjugate to a subshift of finite type}\}
\]
is open, dense and has full Lebesgue measure in $D'$. 
	\end{rem}

\section{Concluding Remarks}\label{examgen}
In this article, we studied the $k$-transformation $T_k$ for integers $k>2$, and examined the intervals $(a,b)$ for which the Hausdorff dimension of $\mathcal{W}_{k}(a,b)$ is positive. The map $T_k$ can be extended to $T_\mathbf{k}=T_{k_1}\times \cdots \times T_{k_r}$ on $[0,1)^r$, with rectangular holes $R_{\mathbf{a},\mathbf{b}}=(a_1,b_1)\times \cdots \times (a_r,b_r)$, where $\mathbf{k}=(k_1,\cdots,k_r)$, $\mathbf{a}=(a_1,\cdots,a_r)$, and $\mathbf{b}=(b_1,\cdots,b_r)$. If 
\[
\mathcal{W}_\mathbf{k}(\mathbf{a},\mathbf{b})=\{\textbf{x}\in [0,1)^r\ \vert\ T_\mathbf{k}^n \mathbf{x}\notin R_{\mathbf{a},\mathbf{b}}, n\geq 0\}, 
\]
then it is easy to see that $\mathcal{W}_{k_1}(a_1,b_1)\times\cdots\times \mathcal{W}_{k_r}(a_r,b_r)\subseteq \mathcal{W}_\mathbf{k}(\mathbf{a},\mathbf{b})$. In general, it is not straightforward to estimate the Hausdorff dimension of $\mathcal{W}_\mathbf{k}(\mathbf{a},\mathbf{b})$. This would be an interesting problem to study.

\section{Acknowledgements}
The research is partially supported by Center for Research on Environment and Sustainable Technologies (CREST), IISER Bhopal.


%

\end{document}